\documentclass[reqno, 11pt]{amsart}
\usepackage{amsmath}
\usepackage{amsfonts}
\usepackage{amssymb}
\usepackage{amsthm}
\usepackage[applemac]{inputenc}
\usepackage[T1]{fontenc}
\usepackage{lmodern}
\title[Potentiels variables]{Potentiels variables et equations dispersives}
\date{Séminaire Laurent Schwartz 2013}
\author[M. Beceanu]{Marius Beceanu}

\newtheorem{theorem}{Théoreme}
\newtheorem{lemma}[theorem]{Lemma}

\newtheorem{observation}[theorem]{Remarque}

\newtheorem{definition}{Définition}
\newcommand{\set}{\mathbb}

\newcommand{\dl}{\nabla}

\renewcommand{\frak}{\mathfrak}

\renewcommand{\upsilon}{R}
\newcommand{\mc}{\mathcal}
\newcommand{\be}{\begin{equation}}
\newcommand{\ee}{\end{equation}}
\newcommand{\bm}{\begin{multline}}
\newcommand{\emm}{\end{multline}}
\newcommand{\ba}{\begin{array}}
\newcommand{\ds}{\displaystyle}
\newcommand{\ea}{\end{array}}
\newcommand{\lb}{\label}

\newcommand{\bpm}{\begin{pmatrix}}
\newcommand{\epm}{\end{pmatrix}}
\newcommand{\dd}{{\, d}}
\newcommand{\les}{\lesssim}
\newcommand{\R}{\set R}
\newcommand{\Z}{\set Z}
\newcommand{\B}{\mc B}

\DeclareMathOperator*{\slim}{s-lim}
\DeclareMathOperator{\ISO}{O}

\begin{document}
\maketitle

\section{Introduction}
Soit l'équation linéaire de Schr\"{o}dinger en $\set R^d$, $d \geq 3$, avec un potentiel réel $V$ dépendant du temps:
\begin{equation}\lb{schroedinger}
i\partial_t \psi(x, t) - \Delta \psi(x, t) + V(x, t) \psi(x, t) = \Psi(x, t),\ \psi(0)=\psi_0.
\end{equation}
Si le potentiel est indépendant du temps, l'équation suit plusieurs lois de conservation, en particulier celles de la masse et de l'énergie:
$$\begin{aligned}
M[\psi] &:= \int_{\R^d} |\psi(x, t)|^2 \dd x\\
E[\psi](t) &:= \int_{\R^d} |\dl \psi(x, t)|^2 + V(x, t) |\psi(x, t)|^2 \dd x.
\end{aligned}
$$
En permettant un potentiel dépendant du temps, la masse est toujours conservée, parce que l'évolution reste unitaire, mais l'énergie peut varier comme fonction du temps et n'est a priori plus bornée.

Si le potentiel est indépendant du temps, la solution a des propriétés dispersives, comme l'inégalité de Strichartz \cite{tao} et la régularité locale (Kato smoothing):
$$\begin{aligned}
\|\psi\|_{L^\infty_t L^2_x} + \|P_c \psi\|_{L^2_t L^{2d/(d-2)}_x} &\les \|\psi_0\|_2\ \text{(Strichartz)}\\
\|D^{1/2} P_c \psi\|_{L^2_t L^2_x(Q)} &\les |Q|^{1/2d} \|\psi_0\|_2\ \text{(régularité locale)}.
\end{aligned}$$
Par ailleurs, on connaît en ce cas-ci que les solutions décroissent à une rate de $t^{-3/2}$ de façon uniforme:
$$
\|P_c \psi(t)\|_{L^\infty_x} \les \|\psi_0\|_{L^1_x}\ \text{(décroissance ponctuelle)}.
$$
Ces inégalités ne tiennent que pour composante dispersive de la solution, déterminée par la projection $P_c$ sur le spectre continu de $-\Delta+V$. En outre, elles ne tiennent que sans états stationnaires d'énergie zéro.

En trois dimensions, l'interaction des états stationnaires d'énergie zéro et du spectre continu génère des termes particuliers qui décroisent à une rate de $t^{-1/2}$, qui doivent être pris en compte d'une manière particulière.

Tout change en principe quand le potentiel est variable. Le but de ce travail-ci est d'énoncer, par la suite, des conditions suffisantes pour garantir les mêmes résultats dans ce cas général.

\section{Résultat principal}
\begin{theorem}
Dans l'équation (\ref{schroedinger}), soit $d=3$ et supposons que\\
i) La famille $\{V(t) : t \geq 0\}$ est precompacte en $\widehat{\dot B^{1/2}_{2, 1}}$ jusqu'aux symmetries, c'est-à-dire isometries et changement d'échelle.\\
ii) Les fonctions $\ds \frac{|x|^{1-\epsilon}}{|x\cdot \omega|^{1-\epsilon}} V(t)$ sont uniformément bornées en $\widehat{\dot B^{1/2}_{2, 1}}$ pour tous $\omega \in S^2$.\\
iii) Pour tout $t \geq 0$, $-\Delta+V(t)$ n'a pas de résonances ou d'états stationnaires d'énergie zéro.\\
iv) $\partial_t V(x, t) \in L^p_t |x|^{\frac 2 p - 2} \widehat{\dot B^{1/2}_{2, 1}}_x$ pour $p \in [1, 2)$ donné.\\
Alors on a que
$$
\|P_c(t) \psi(t)\|_{L^2_t L^{6, 2}_x} \les \|\psi_0\|_{L^2_x} + \|\Psi\|_{L^2_t L^{6/5, 2}_x}.
$$
\end{theorem}

Ces estimations dispersives sont valables dans un régime non-perturbatif et ne sont pas réductibles (d'une manière habituelle) au cas d'un potentiel fixe. En particulier, $V(x, t)$ n'a pas forcément une limite quand $t \to \infty$.

Remarques:

1) L'espace de Banach $\widehat{\dot B^{1/2}_{2, 1}} \equiv |x|^{-1/2} \widehat{\dot B^0_{2, 1}}$ est défini par sa norme
$$
\|V\|_{\widehat{\dot B^{1/2}_{2, 1}}} := \sum_{k \in \Z} 2^{k/2} \|\chi_{[2^k, 2^{k+1})}(|x|) V(x)\|_{L^2_x}.
$$
C'est comparable à $\langle x \rangle^{-1/2-\epsilon} L^2$, mais $\widehat{\dot B^{1/2}_{2, 1}}$ est en plus invariant par rapport aux changements d'échelle:
$$
\|\alpha^2 V(\alpha \, \cdot)\|_{\widehat{\dot B^{1/2}_{2, 1}}} = \|V\|_{\widehat{\dot B^{1/2}_{2, 1}}}
$$
pour chaque $\alpha = 2^k$, $k \in \Z$.

2) $-\Delta+V(t)$ a un nombre constant de valeurs propres: celui-ci ne change que si un état stationnaire d'énergie negative passe par énergie zéro, ce que nos hypothèses interdisent.

3) Ajouter un petit potentiel variable ($L^\infty_t L^{3/2}_x$) à l'équation ne change pas le résultat, pourvu qu'on contrôle les états stationnaires $P_p(t) \psi(t)$.

4) Ce résultat s'améliore en dimension supérieure $d>3$, mais sous conditions supplémentaires sur $P_p(t) \psi(t)$. Si ces conditions sont accomplies, il devrait être possible de montrer des inégalités de Strichartz pour
$$
\partial_t V \in L^p_t L^p_t |x|^{\frac 2 p - 2} \widehat {\dot B^{\frac {d-2}{d-1}}_{\frac{d-1}{d-2}, 1}}
$$
avec $1 \leq p<\infty$ en dimension $d=4$ et même $1\leq p \leq \infty$ pour $d \geq 5$.

Si $p<\infty$ (ce qui est toujours le cas en $\R^3$ et probablement en $\R^4$), on gagne une petite constante en sous-divisant l'interval $[0, \infty)$ ; quand $p=\infty$ il faut prendre la dérivée $\partial_t V$ petite et non seulement finie.

5) De tels résultats devraient aussi être valables pour l'équation des ondes.

\section{Historique du problème}
Les résultats connus pour l'équation de Schr\"{o}dinger à potentiel variable sont plutôt de deux types, selon qu'il s'agit de variétés compactes ou non.

Sur les variétés compactes il n'y a pas de dispersion, donc on cherche la rate de croissance de l'énergie. Bourgain \cite[1999]{bou1}, \cite[1999]{bou2}, \cite[2003]{bou3} montra que, pour des potentiels lisses sur le tore, l'énergie croît au de façon au moins logarithmique, mais moins vite que $t^\epsilon$ pour tout $\epsilon>0$. Spencer pour des potentiels périodiques et Bourgain pour des potentiels quasi-périodiques montrèrent que la rate de croissance est tout au plus logarithmique.

En 2008, Wang \cite{wan} montra le même pour des potentiels analytiques, résultat de nouveau généralisé par Fang--Zhang \cite[2012]{fazh} au cas des potentiels de type Gevrey et aux dimensions supérieures.

Le cas d'autres variétés compactes reste pour la plupart ouvert ; Delort \cite[2010]{del} traite le cas sphérique et énonce une condition plus générale sur les valeurs propres du laplacien sur une variété qui permettrait de retrouver une borne logarithmique pour la rate de croissance.

Pour le cas d'une variété noncompacte, en espèce pour un potentiel périodique en temps sur $\R^3$, Galtbayar--Jensen--Yajima \cite[2002]{gjy} établirent une rate de décroissance de $t^{-3/2}$ pour le cas non-résonant et de $t^{-1/2}$ dans le cas résonant. Par la suite, en 2003 Rodnianski--Schlag \cite{rodsch} montrèrent une rate de décroissance ponctuelle de $t^{-3/2}$, pour de petits potentiels variables.

Goldberg \cite{gol}, encore pour des potentiels périodiques en temps, montra en 2009 l'inégalité de Strichatz sous une condition impliquant l'absence des états stationnaires d'énergie zéro.

Costin--Lebowitz--Tanveer \cite[2010]{clt}, pour un potentiel à longue portée sur $\set R^3$ avec une perturbation sinusoïdale périodique, montrèrent l'absence des états bornés de l'opérateur de Floquet, vérifiant ainsi l'hypothèse spectrale, et trouvèrent une rate ponctuelle de décroissance de $t^{-5/3}$.

Beceanu \cite[2011]{bec2} montra que, pour un potentiel à profil fixe évoluant sous l'action d'une famille d'isométries, l'inégalité de Strichartz reste vraie, si le paramètre de modulation a une dérivée en $L^1_t$ (dans l'absence des états d'énergie zéro). Beceanu--Soffer \cite[2011]{becsof1} assouplira cette condition en admettant que le paramètre de modulation soit en $\dot H^{1/2} \cap C$ ; on ajouta aussi les changements d'échelle aux isometries permissibles.

\section{Esquisse de la demonstration}
\subsection{Idée de la demonstration}
L'idée principale est que, moyennant les opérateurs d'onde, on peut éliminer entièrement le potentiel variable.

Les opérateurs d'onde $W_{\pm}(t)$ ont la propriété fondamentale de jumelage entre le hamiltonien perturbé et le hamiltonien libre:
$$
W_{\pm}(t) (-\Delta+V(t)) = -\Delta W_{\pm}(t).
$$

Appliquant $W_+(t)$ à l'équation et mettant $W_+(t) \psi(t) = \phi(t)$, on obtient l'équation transformée
\be\lb{eq_trans}
i\partial_t \phi(t) - \Delta \phi(t) = W_+(t)\Psi(t) + i \partial_t W_+(t) W_+^*(t) \phi(t) + i \partial_t W_+(t) P_p(t) \psi(t).
\ee

C'est plus facile de démontrer l'inégalité de Strichartz pour $\phi$, à condition qu'on sache contrôler les termes d'erreur résultant du potentiel variable.

Pour qu'on puisse aller et revenir entre les deux équations, il faut que les opérateurs d'onde soient uniformément bornés sur $L^p$, par example pour $6/5 \leq p \leq 6$. C'est ce que garantit, par exemple, la première condition de notre hypothèse --- que nous pourrions même remplacer en demandant le bornage des opérateurs d'onde.

Le dernier terme $i \partial_t W_+(t) P_p(t) \psi(t)$ en (\ref{eq_trans}) est contrôlé par la norme $L^2$ (conservée) ; il faut que $\partial_t W_+(t)$ soit en $L^2_t$ ou qu'une condition de modulation soit satisfaite.

Pour l'autre terme d'erreur $i \partial_t W_+(t) W_+^*(t)$, il faut améliorer la rate de décroissance en $x$, par example $L^3 \mapsto L^{3/2}$ donnerait $L^2_t$:
$$
L^{4/3}_t L^{3/2}_x \xrightarrow{Strichartz} L^4_t L^3_x \xrightarrow{\partial_t W_+(t) W_+^*(t)} L^{4/3}_t L^{3/2}_x.
$$
Idéalement, il faudrait deux puissances de décroissance en $x$ pour gagner une puissance de décroissance en $t$ (impossible en $\R^3$):
$$
L^2_t L^{6/5}_x \xrightarrow{Strichartz} L^2_t L^6_x \xrightarrow{\partial_t W_+(t) W_+^*(t)} L^2_t L^{6/5}_x.
$$
Sans décroissance, on a besoin de $L^1_t$, ce qui serait pareil aux résultats préexistants.

Notons que le gain en décroissance demeure possible pour $\partial_t W_+$ parce qu'en prenant la dérivée le terme principal (l'identité) devient zéro.

\subsection{Les opérateurs d'onde}
Ils sont définis par les formules
$$
W_{\pm} := \slim_{t \to \pm \infty} e^{it (-\Delta+V)} e^{it\Delta}
$$
$$
W^*_{\pm} = \slim_{t \to \pm \infty} e^{-it \Delta} e^{-it(-\Delta+V)} P_c
$$
La complétude asymptotique est une propriété fondamentale des opérateurs d'onde qui est satisfaite sous des conditions assez générales.
\begin{definition}
On dit que les opérateurs d'onde sont asymptotiquement complets si
\begin{enumerate}
\item[i] $W_{\pm}$ sont $L^2$-bornés et surjectifs de $L^2$ à $P_c L^2$.
\item[ii] Le spèctre singulier continu de $H$ est vide, $\sigma_{sc}(H) = \varnothing$.
\end{enumerate}
\end{definition}
Un résultat fondamental de Agmon en 1975 \cite{agmon}, appuyé sur des contributions de Kato et de Kuroda, dit que ceci est le cas si $V \in \langle x \rangle^{-1-\epsilon} L^{\infty}$.

En d'autres circonstances, on souhaite plutôt une condition de décroissance en moyenne sur $V$. Un tel résultat est impliqué de manière triviale par les inégalités de Strichartz si $V \in L^{3/2}$.

En outre, Ionescu--Schlag \cite[2006]{iosc} établirent la complétude asymptotique des opérateurs d'onde sous une variété beaucoup plus large de conditions, permettant aussi des potentiels magnétiques.

Dans cette situation, $W_\pm$ sont des isometries partielles entre $L^2$ et $P_c L^2$, même unitaires s'il n'y a pas d'états stationnaires:
$$
W_\pm^* W_\pm = I,\ W_\pm W_\pm^* = P_c.
$$
Par ailleurs, grâce aux propriétés de régularité elliptique, $W$ préservent jusqu'à deux dérivées sans conditions supplémentaires.

Pour d'autres éspaces $L^p$ si $p \ne 2$, le premier tel résultat appartient à Yajima \cite[1993]{yajima}, qui démontra que les opérateurs d'onde $W_\pm$ et leurs adjoints sont bornés sur $L^p$, $1 \leq p \leq \infty$, en dimension $d \geq 3$ dans l'absence des états stationnaires d'énergie zéro.

Ce résultat fut suivi par d'autres, pour la plupart toujours obtenus par Yajima et ses collaborateurs. En dimension $d=2$, en particulier, ils montrèrent que les opérateurs d'onde ne sont bornés en général que pour $1<p<\infty$.

En 2012, Beceanu \cite{bec} montra que les opérateurs d'onde sont bornés sous des conditions plus faibles sur $V$ et obtenit une formule de structure plus précise.

Soit $\ISO(3) = \{s \in \B(\R^3, \R^3) \mid s^* s = I\}$ le groupe des transformations linéaires orthogonales, c'est-à-dire des isométries (linéaires), sur $\R^3$.

On obtenit le résultat suivant:
\begin{theorem}\lb{theorem_1.1}
Soit $V \in \widehat {\dot B^{1/2}_{2, 1}}$ un potentiel à valeurs réelles tel que $H = -\Delta+V$ n'admet pas de fonction propre ou résonance à énergie zéro. Alors pour chacun d'entre $W_{\pm}$ et $W_{\pm}^*$ il existe $g_{s, y}(x) \in (\mc M_{loc})_{s, y, x}$ tel que $\|g_{y, s}(x)\|_{L^{\infty}_x} \in L^1_y \mc M_s$, c'est-à-dire
$$
\int_{\set R^3} \Big(\int_{\ISO(3)} \dd \|g_{s, y}\|_{L^{\infty}_x}\Big) \dd y < \infty
$$
et pour $f \in L^2$ on a la formule de répresentation
\be\begin{aligned}\lb{eqn1.8}
(W f)(x) &= f(x) + \int_{\R^3} \Big(\int_{\ISO(3)} \dd g_{s, y}(x) f(sx + y)\Big) \dd y.
\end{aligned}\ee
\end{theorem}
En effet, les fonctions $g$ sont sont des combinations linéaires intégrables de fonctions caractéristiques de démi-éspaces, fait qui importe plutôt en dimension supérieure.

Par consequence, en dimension supérieure $d>3$ on retrouve la même formule de structure pour $V \in \widehat {\dot B^{\frac {d-2}{d-1}}_{\frac{d-1}{d-2}, 1}}$. Comme $\frac{d-1}{d-2}=2$ seulement pour $d=3$, le théorème de Plancherel ne s'applique plus quand $d>3$ et il faut se contenter avec une representation en fréquence.

Les opérateurs d'onde ont une expansion asymptotique par la formule de Duhamel. Le premier terme est l'identité, puis le suivant est
$$
W f := W_1 f := i \int_0^{\infty} e^{-it\Delta} V e^{it\Delta} f \dd t,
$$
puis plusieurs termes pareils:
$$\begin{aligned}
W_n f = i^n \int_{0\leq t_1 \leq \ldots \leq t_n} e^{i(t_n-t_{n-1})H_0} V \ldots e^{i(t_2 - t_1)H_0} V e^{it_1 H_0} V e^{-it_nH_0} f \dd t_1 \ldots \dd t_n.
\end{aligned}$$
La série avec ces terms converge en norme pour $V$ petit.

\begin{observation} Ecrit comme opérateur pesudodifférentiel en fréquence, le deuxième terme a le noyau intégral
$$
\widehat W = \frac {\widehat V(\xi_1-\xi_2)}{|\xi_1|^2-|\xi_2|^2}.
$$
Combien de décroissance peut-on gagner? En regardant ailleurs que la diagonale, par exemple dans la région où $|\xi_1|>2|\xi_2|$, l'opérateur se réduit~à
$$
\chi_{|\xi_1|>2|\xi_2|} \widehat W \sim \chi_{|\xi_1|>2|\xi_2|} \frac {\widehat V(\xi_1-\xi_2)}{|\xi_1|^a |\xi_2|^{2-a}} = \chi_{|\xi_1|>2|\xi_2|} |\dl|^{-a} V |\dl|^{a-2}.
$$
Même dans ce cas optimal, on gagne une puissance de décroissance tout au plus: $L^{3, 1} \mapsto L^{3/2, \infty}$.

On peut aussi examiner l'équation de la chaleur:
$$
\partial_t \psi(x, t) + \Delta \psi(x, t) - V(x, t) \psi(x, t) = \Psi(x, t),\ \psi(0)=\psi_0.
$$
Celle-ci est plus facile à étudier et a des propriétés similaires. On utilise le fait que $W \in \B(\dot H^{1/2}, \dot H^{-1/2})$, de même que $W_\pm$ et leurs adjoints, si $V \in L^1$. L'inégalité de Morawetz est importante en ce cas-limite.

Au vu du gain de régularité dans
$$
\|\int_{t>s} e^{(t-s)\Delta} F(s) \dd s\|_{L^4_t \dot H^{1/2}_x} \les \|F\|_{L^{4/3}_t \dot H^{-1/2}_x},
$$
ça vaut une puissance de décroissance en $x$ pour l'équation de la chaleur. C'est-à-dire qu'on peut prendre $\partial_t V \in L^2_t L^1_x$, un peu mieux que pour l'équation de Schr\"{o}dinger.

En revanche, comme l'évolution n'est plus unitaire, il faudrait procéder différemment pour controller les états stationnaires.
\end{observation}

Retournant à la demonstration du résultat principal, une formule alternative qu'on utilisera pour le deuxième terme est
$$\begin{aligned}
W f(x) = \int_{S^2} \int_{[0, \infty)} &K(x, t \omega) f(x+t \omega) \dd t \dd \omega,
\end{aligned}$$
où $K$ s'écrit dans des coordonnées polaires comme
$$\begin{aligned}
K(x, t\omega) = \frac 12\int_{[0, \infty)} \widehat {V}(s \omega) e^{-it s/2} e^{i s \omega \cdot x} s \dd s.
\end{aligned}$$
Soit
$$
L_\pm(t, \omega) = \int_{[0, \pm\infty)} \widehat {V}(s \omega) e^{-it s} s \dd s.
$$
Ainsi, $x \mapsto x - 2(x \cdot \omega) \omega$ étant une isométrie, on écrit $K$ comme
$$\begin{aligned}
K(x, t \omega) &= \frac 12\chi_{(-\infty, \frac {t} 2)}(x \cdot \omega)\ L_+((t-2x\cdot \omega) \omega) +\\
&+ \frac 12\chi_{(\frac {t} 2, \infty)}(x \cdot \omega)\ L_-((t - 2x \cdot \omega)\omega).
\end{aligned}$$
Cette représentation est conforme à la formule de représentation (\ref{eqn1.8}), pourvu que $L_\pm$ est intégrable:
\be\lb{Lpm}
\int_{S^2} \int_{\R} |L_\pm(t, \omega)| \dd t \dd \omega < \infty.
\ee
C'est bien le cas, au moins si $V \in \widehat{\dot B^{1/2}_{2, 1}}$:
\be\lb{estim}
\int_{S^2} \int_{\R} |L_\pm(t, \omega)| \dd t \dd \omega \les \|V\|_{\widehat{\dot B^{1/2}_{2, 1}}}.
\ee
En effet, on a que $L_\pm \in L^2_\omega$ sans condition supplémentaire.

\subsection{Rappel} Récapitulant Beceanu \cite{bec}, on voudrait maintenant écrire les autres termes comme puissances du premier terme, de sorte que leurs propriétés suivent à celles du premier terme. En plus, la série serait ainsi une série géométrique, qu'on saurait additionner.

Ceci est bien le cas, mais il faut lifter ces opérateurs pseudo-différentiels pour voir la structure d'algèbre de Banach. La loi d'algèbre est
$$
(T_1 \ast T_2)(x_0, x_2, y) = \int_{\set R^6} T_1(x_0, x_1, y_1) T_2(x_1, x_2, y-y_1) \dd x_1 \dd y_1.
$$
Cette algèbre, qu'on appelle $Y$, est définie comme
\be\lb{2.31}\begin{aligned}
Y &:= \{T(x_0, x_1, y) \in \mc S' \mid \forall s \in \ISO(3)\ T(x_0, x_1, y + x_0-s x_0) \in \mc S', \\
&\forall g \in L^{\infty}\ \forall s \in \ISO(3)\ \int_{\R^3} g(x_0) T(x_0, x, y+x_0-s x_0) \dd x_0 \in X\}.
\end{aligned}\ee
L'espace $X$ est l'espace de noyaux intégraux à deux variables conformes à la formule de structure (\ref{eqn1.8}):
\be\begin{aligned}\lb{2.9}
X := &\{\frak X \in \B(L^{\infty}, L^{\infty})\mid (\frak X f)(x) = \int_{\R^3} \frak X(x, y) f(x-y)  \dd y, \\
&\frak X(x, y) = \int_{\ISO(3)} g_{s, y + x - s x}(x) \dd s, \\
&\int_{\R^3} \int_{\ISO(3)} \dd \|g_{s, y}\|_{L^{\infty}_x} < \infty\},
\end{aligned}\ee

Le lifting du premier terme
$$
\widehat W = \frac {\widehat V(\xi)}{|\xi+\eta|^2-|\eta|^2}
$$
vers l'algèbre $Y$ est donné par
$$
\widehat T(\xi_0, \xi_1, \eta) = \frac {\widehat V(\xi_1-\xi_0)}{|\xi_1+\eta|^2-|\eta|^2}
$$
et l'opération inverse (de projection) est de remettre $\xi_0=0$.

Alors chaque terme $W_n$ est la projection de la puissance $T^n$ de $T$ prise dans cette algèbre $Y$: ça revient à
$$\begin{aligned}
(\mc F_{x_0, x_n, y} \widehat T^n(\xi_0, \xi_n, \eta) &:= 
\int_{\R^{3(n-1)}} \frac{\prod_{\ell=1}^n \widehat V(\xi_{\ell} - \xi_{\ell-1}) \dd \xi_1 \ldots \dd \xi_{n-1}}{\prod_{\ell=1}^n (|\xi_{\ell}+\eta|^2-|\eta|^2 - i 0)}
\end{aligned}$$
et puis on projette en prenant $\xi_0=0$ pour retrouver $W_n$.
 
D'abord on montre que $\widehat T \in Y$, ce qui revient à verifier (\ref{Lpm}) pour $V$ tronqué par la fonction caractéristique d'un démi-espace quelconque, c'est-à-dire avec $\chi_{\{x \cdot \omega_0 \geq t_0\}}(x) V(x)$ remplaçant $V$. C'est vrai parce que
$$
\|\chi_{\{x \cdot \omega_0 \geq t_0\}}(x) V(x)\|_{\widehat {B^{1/2}_{2, 1}}} \leq \|V\|_{\widehat {B^{1/2}_{2, 1}}}.
$$

Si $V$ est petit la série converge, sinon on obtient $(I+T)^{-1} \in Y$ par le théorème de Wiener ;  voir Beceanu \cite{bec} pour ces calculs.

\subsection{Suite} En tout cas, maintenant prendre la dérivée comme fonction du temps revient à estimer
$$
\partial_t (I+T)^{-1} = -(I+T)^{-1} \partial_t T (I+T)^{-1}
$$
et puis projeter en prenant $\xi_0=0$.

Tout d'abord il faut comprendre $\partial_t T$. On commence par
$$
\partial_t K(x, t\omega) = \frac 12\int_{[0, \infty)} \widehat {V_t}(s \omega) e^{-it s/2} e^{i s \omega \cdot x} s \dd s.
$$
Pour obtenir une décroissance en $x$, il faut intégrer par parts en $s$. En ce faisant on gagne $\ds\frac 1{|x\cdot \omega|}$, mais en perdant $\frac 1 s$, $t$, and $\partial_s$: pour $V_t \in \mc S$
\be\begin{aligned}\lb{int_part}
\partial_t K(x, t\omega) = -\frac 1{2 x \cdot \omega} \int_{[0, \infty)} &\big(\partial_s \widehat {V_t}(s \omega) e^{-it s/2} s - it \widehat {V_t}(s \omega) e^{-it s/2} s + \\
&+ \widehat {V_t}(s \omega) e^{-it s/2} \big) e^{is \omega \cdot x} \dd s.
\end{aligned}\ee
Les premiers deux termes sont pareils à (\ref{estim}), mais avec $V$ substitué par $\partial_s \widehat {V_t}(s \omega) + \widehat {V_t}(s \omega)$, donc ce sont contrôlés par
$$
\|\partial_s \widehat {V_t}(s \omega) + \widehat {V_t}(s \omega)\|_{\dot B^{1/2}_{2, 1}}.
$$
L'inégalité de Hardy assure que
$$
\|\frac {\widehat V} s\|_{\dot B^{1/2}_{2, 1}} \les \|\partial_s \widehat V\|_{\dot B^{1/2}_{2, 1}},
$$
donc cette dernière expression est une borne pour ces termes.

Pour le troisième terme $\widehat {V_t}(s \omega)$ il faudrait gagner un facteur de $t^{-1}$, mais c'est impossible: l'intégration par parts rend $t^{-2}$, qui n'est pas intégrable après la multiplication par $t$.

On retrouve donc le résultat moins un epsilon et c'est pourquoi il ne tient pas dans le cas-limite.

\begin{lemma} Pour $0<\epsilon\leq 1$
$$
\||x \cdot \omega|^{1-\epsilon} \int_{[0, \infty)} \widehat {V}(s \omega) e^{-it s/2} e^{i s \omega \cdot x} s \dd s\|_{L^\infty_x L^1_{t, \omega}} \les \|\widehat V\|_{\dot B^{3/2-\epsilon}_{2, 1}}.
$$
\end{lemma}
\begin{proof} On montre d'abord que pour chaque $k \in \Z$ et $x \in \R^3$
\be\lb{7}
\|\chi_{2^k \leq |t| \leq 2^{k+1}}(t) \int_{[0, \infty)} \widehat {V}(s \omega) e^{-it s/2} e^{i s \omega \cdot x} s \dd s\|_{L^1_{t, \omega}} \les |x \cdot \omega|^{-1}  \|\widehat V\|_{\dot B^{3/2}_{2, 1}}
\ee
(donc il y a une perte logarithmique pour $\epsilon=0$), aussi que
\be\lb{8}
\sum_{k \in \Z} \|\chi_{2^k \leq |t| \leq 2^{k+1}}(t) \int_{[0, \infty)} \widehat {V}(s \omega) e^{-it s/2} e^{i s \omega \cdot x} s \dd s\|_{L^1_{t, \omega}} \les \|\widehat V\|_{\dot B^{1/2}_{2, 1}}.
\ee

Pour montrer (\ref{7}), il suffit de supposer que $V(x)$ est supporté dans une région compacte comme $\{x \mid 2^k \leq |x| \leq 2^{k+1}\}$ --- et de prendre $k=1$ sans perte de généralité. Alors $\widehat V$ est analytique et en particulier est en $H^2$ ;  la conclusion suit par (\ref{int_part}).

Enfin, (\ref{8}) est ce qu'on savait déjà par (\ref{estim}) et correspond à $\epsilon=1$.

Alors, pour $x \in \R^3$ et $\epsilon \in (0, 1)$ on obtient par interpolation réelle que
$$
\sum_k \|\chi_{2^k \leq |t| \leq 2^{k+1}}(t) \int_{[0, \infty)} \widehat {V}(s \omega) e^{-it s/2} e^{i s \omega \cdot x} s \dd s\|_{L^1_{t, \omega}} \les |x \cdot \omega|^{-1+\epsilon} \|V\|_{\dot B^{3/2-\epsilon}_{2, 1}}.
$$
\end{proof}

Il faut alors utiliser cette décroissance, qui tient dans une seule direction à la fois. Une possibilité est de prendre la moyenne
$$
\int_{S^2} \Big\|\frac {f(x)}{|x\cdot \omega|^{1-\epsilon}}\Big\|_{L^1_x} \dd \omega \les \|f\|_{L^{3/2-\epsilon, 1}_x}.
$$
Il n'y a aucun gain quand $f \in L^\infty$, mais par interpolation on en obtient un pour $f \in L^p$, $p<\infty$. Celui-ci n'est optimal que pour $p=1$, donc
on en perd de cette façon.

Pour retrouver le gain optimal, une approche est d'utiliser des inégalités de Strichartz retardées hétérogènes.

\begin{lemma}
Pour toutes les directions $\omega_1$ et $\omega_2 \in S^2$, en décomposant $\R^3 = \R_{\omega_1} \oplus \R_{\omega_1}^\perp = \R_{\omega_2} \oplus \R_{\omega_2}^\perp$, on a
$$
\|\int_{t>s} e^{-i(t-s)\Delta} F(s) \dd s\|_{L^4_t L^\infty_{\omega_2} L^2_{\omega_2^\perp}} \les \|F\|_{L^{4/3}_t L^1_{\omega_1} L^2_{\omega_1^\perp}}.
$$
\end{lemma}
\begin{proof}
C'est facile, d'abord pour $\omega_1=\omega_2$, d'obtenir (par la méthode habituelle de l'intégration fractionnelle) que
$$
\bigg\|\int_{t>s} e^{-i(t-s)\Delta} F(s) \dd s\bigg\|_{L^4_t L^\infty_{\omega_1} L^2_{\omega_1^\perp}} \les \|F\|_{L^{4/3}_t L^1_{\omega_1} L^2_{\omega_1^\perp}}
$$
et de même pour $t<s$. Alors, par la méthode $T T^*$ on obtient que
$$
\|e^{-it\Delta} f\|_{L^4_t L^\infty_{\omega_1} L^2_{\omega_1^\perp}} \les \|f\|_{L^2_x}.
$$
En passant par $L^2$ de cette façon on obtient que
$$
\bigg\|\int_{\R} e^{-i(t-s)\Delta} F(s) \dd s\bigg\|_{L^4_t L^\infty_{\omega_2} L^2_{\omega_2^\perp}} \les \|F\|_{L^{4/3}_t L^1_{\omega_1} L^2_{\omega_1^\perp}}.
$$
Puis on utilise le lemma de Christ--Kiselev pour obtenir l'inégalité retardée qu'on désirait:
\begin{lemma}[\cite{chki}]
Pour deux espaces de Banach $X$ and $Y$, soit l'opérateur borné $T:L^p(X) \to L^q(Y)$, donné par son noyau intégral
$$
(T F)(t) = \int_\R T(t, s) F(s) \dd s.
$$
Si $1 \leq p < q < \infty$, alors l'opérateur $\tilde T$ de noyau $\chi_{t>s} T(t, s)$ est aussi borné de $L^p(X)$ à $L^q(Y)$.
\end{lemma}
On peut utiliser ce lemma parce qu'on est loin du cas-limite, c'est-à-dire $4/3<4$.
\end{proof}

\end{document}